\documentclass{article}
\usepackage{t1enc}
\usepackage{latexsym}
\usepackage{amssymb}
\usepackage{amsmath}
\usepackage{amsthm}
\usepackage[mathscr]{eucal}
\usepackage{graphicx}
\usepackage{pictexwd}
\usepackage{dcpic}

\newcommand{\bibname}{}

\newtheorem{thm}{Theorem}
\newtheorem{lem}{Lemma}

\theoremstyle{definition}
\newtheorem{dfn}{Definition}
\newtheorem{rem}{Remark}

\newcommand{\Mor}[2]{\ensuremath{\mbox{Cob}_{\Sigma^1}(#1,#2)}}
\newcommand{\tps}{\ensuremath{tp_{\Sigma^2}}}
\newcommand{\MorQ}{\ensuremath{\mbox{Mor}_{\Q}} }
\newcommand{\Prim}[2]{\ensuremath{\mbox{Prim}(#1,#2)}}
\newcommand{\defeq}{\stackrel{\mbox{\scriptsize def}}{=}}
\newcommand{\xih}{\tilde{\xi}}

\newcommand{\R}{\mathbb{R}}
\newcommand{\Z}{\mathbb{Z}}
\newcommand{\Q}{\mathbb{Q}}

\newcommand{\hu}[3]{\tilde{#1}_{#2}(#3)}

\newcommand{\Sr}[1]{\Sigma^{1_{#1}}}
\newcommand{\id}{\mbox{\rm{id}}}
\newcommand{\rank}{\mbox{rank }}
\newcommand{\ie}{i{.} e{.} }

\title{Multiplicative properties of Morin maps}
\author{G\'abor Lippner \and Andr\'as Sz\H ucs}
\begin{document}
\maketitle

\begin{abstract}
In the first part of the paper we construct a ring structure on the rational
cobordism classes of Morin maps (\ie smooth generic maps of corank 1). We
show that associating to a Morin map its $\Sr{r}$ (or $A_r$) singular strata
defines a ring homomorphism to $\Omega_* \otimes \Q$, the rational oriented
cobordism ring. This is proved by analyzing the multiple-point sets of a
product immersion. Using these homomorphisms we compute the ring of Morin maps.

In the second part of the paper we find the oriented Thom polynomial of the
$\Sigma^2$ singularity type with $\Q$ coefficients. Then we provide a product
formula for the $\Sigma^2$ and the $\Sigma^{1,1}$ singularities.

\end{abstract}

\section{Introduction}

The results of this paper are the first steps toward understanding how the
direct product operation affects the singularities of maps. There are two main
problems. The first one is that the direct product of generic maps will not be
generic, so one has to take a small perturbation. This makes it hard to
understand the singular strata geometrically. The second problem is that in
general the product of two singular maps even after a generic perturbation will
have more complicated singularities then the original maps had.

In Section~\ref{immerzioresz} we study products of immersions. Here only the
first type of problem arises, namely that the self intersections will not be
transverse. This can be overcome by employing a general multiple-point formula
from~\cite{BL} that helps to compute the characteristic numbers of
multiple-point manifolds.

In Section~\ref{morinresz} we study Morin maps. In this case one has to deal
with the second kind of problem. We get around this by increasing the dimension
of the target space by one.

In Section~\ref{morq} we set out to compute the ring \MorQ (the ring of
rational cobordism classes of Morin maps) defined at the end of
Section~\ref{morinresz}. First, in Section~\ref{homom}, combining the results
of the previous sections we show that the singular strata behave nicely under
the multiplication defined in Section~\ref{szorzas}. Then in
Section~\ref{direktosszeg} we show that this information is actually enough to
compute \MorQ.

Finally Section~\ref{szigma1} deals with general singular maps. We show that a
Cartan-type formula relates the $\Sigma^1$ points of two maps with the
$\Sigma^1$ points of their direct product. We compute the oriented
Thom polynomial of the $\Sigma^2$ singularity with $\Q$ coefficients. Finally
we derive a Cartan-type formula for the $\Sigma^2$ points as well.

\section{Products of immersions}~\label{immerzioresz}

We start this section by recalling some basic notions about multiple points and
the relevant results of \cite{BL}.

First we shall introduce a characteristic class $\beta$ that assigns
to any oriented vector bundle $\xi$ over $B$ an element
\[ \beta(\xi) = \prod_{i=1}^{\infty}(1+p_1(\xi)t_i+p_2(\xi)t_i^2+\dots)
\in H^*(B;\Q)[[t_1,t_2,\dots]] \] in the ring of formal power series
of the variables $t_i$ over the ring $H^*(B;\Q)$. (Here $p_i(\xi) \in
H^{4i}(B;\Q)$ is the $4i$-dimensional Pontrjagin class of
$\xi$). Since the Cartan formula holds for Pontrjagin classes modulo
2-torsion it follows that $\beta(\xi \oplus \eta) = \beta(\xi) \cdot
\beta(\eta)$. (We have got rid of all torsions by taking $\Q$
coefficients.) It is also easily seen that $\beta$ is natural, and
always has an inverse element. When $B$ is a manifold we shall
abbreviate $\beta(TB)$ by $\beta(B)$.

Now let $f: M^n \to N^{n+k}$ be a generic (i{.}e{.} selftransverse)
immersion between oriented manifolds. The manifolds and
the maps representing the \textit{r-fold points} of $f$ in the source
and the target respectively will be denoted by

\begin{eqnarray*} \phi_r(f)& : \hu{M}{r}{f} \to M,& \mbox{ and} \\
\psi_r(f)& : \hu{N}{r}{f} \to N.& \end{eqnarray*}

When the codimension of the map $k$ is even, these manifolds are equipped with
a natural orientation.  It is easy to see that the cobordism classes of these
manifolds depend only on the cobordism class of $f$. Our goal is to obtain
information about these cobordism classes. To this end we compute their
characteristic numbers.

Let us denote
\begin{eqnarray*}
m_r& = m_r(f)&= {\phi_r}(f)_! (\beta(\hu{M}{r}{f})),\\
n_r& = n_r(f)&= {\psi_r}(f)_! (\beta(\hu{N}{r}{f})).
\end{eqnarray*}

The reason for considering these elements is the following simple
observation. Evaluating each coefficient of $m_r$ on the fundamental
class of $M$ we get an element in $\Q[[t_1,t_2,\dots]]$. The
coefficients of this power series are exactly the Pontrjagin numbers
of $\hu{M}{r}{f}$.

The classes $m_r$ and $n_r$ are related by the equality:
\begin{equation}~\label{mf} m_r \cdot \beta(\nu_f) = f^* n_{r-1} - e(\nu_f)m_{r-1}
\end{equation}
 where $\nu_f$ is the normal bundle of $f$ and $e$ is the Euler
 class. This is a generalization of the well-known Herbert-Ronga
 formula (see the Main formula of~\cite{BL}).

We are going to apply this in the case when the target is a Euclidean space.
Then $f^*=0$ so (\ref{mf}) is simplified to $m_r \cdot \beta(\nu_f) = -
e(\nu_f)\cdot m_{r-1}$. Applying this recursively one gets that $m_r \cdot
\beta(\nu_f)^{r-1} = (-e(\nu_f))^{r-1} \cdot m_1$. But $m_1 = \beta(M)$ and
$\beta(M) \cdot \beta(\nu_f) = \beta(\R^n) = 1$, so we end up with
\[ m_r = (-e(\nu_f))^{r-1}\cdot \beta(M)^r.\]

Now we can state and prove the main result of this section.

\begin{thm}~\label{immszorz}
Let $g_i : M_i^{n_i} \to \R^{n_i+k_i}; (i=1,2)$ be generic immersions. Then we
have \begin{equation}~\label{szorzodik} \hu{M}{r}{g_1\times g_2} \sim
(-1)^{r-1}\hu{M}{r}{g_1} \times \hu{M}{r}{g_2}\end{equation} where $\sim$
stands for ``unoriented-cobordant''.

If both manifolds $M_i$ are oriented and both codimensions $k_i$ are even, then
the two sides of~\ref{szorzodik} are oriented cobordant.
\end{thm}

\begin{proof}
We will only consider the oriented case. The unoriented version is proved
exactly the same way, except that there is no need to study Pontrjagin classes.

Let $f = g_1 \times g_2$. Then

\begin{eqnarray*}
m_r(f) =&& (-e(\nu_f))^{r-1} \cdot \beta((M_1\times M_2))^r = \\
=&& (-e(\nu_{g_1}\times \nu_{g_2}))^{r-1} \cdot \beta(TM_1 \times TM_2)^r = \\
=&& (-1)^{r-1}\left((-e(\nu_{g_1})^{r-1}\cdot \beta(M_1)^r \right)\times
\left((-e(\nu_{g_2})^{r-1}\cdot \beta(M_2)^r \right) =
\\ =&&
(-1)^{r-1} m_r(g_1)\times m_r(g_2).
\end{eqnarray*}

The following equations are easily checked.

\begin{eqnarray*}
\langle \beta(\hu{M}{r}{f}),[\hu{M}{r}{f}]\rangle =&& \langle m_r(f),[M_1\times
M_2]\rangle = \langle m_r(g_1\times g_2),[M_1\times M_2]\rangle =\\=&&
(-1)^{r-1}\langle \beta(\hu{M}{r}{g_1},[\hu{M}{r}{g_1}]\rangle \cdot \langle
\beta(\hu{M}{r}{g_2},[\hu{M}{r}{g_2}]\rangle = \\ =&& (-1)^{r-1}\langle
\beta((\hu{M}{r}{g_1}\times \hu{M}{r}{g_2})),[\hu{M}{r}{g_1}\times
\hu{M}{r}{g_2}]\rangle
\end{eqnarray*}

We have obtained equality of two formal power series, so the corresponding
coefficients must be equal on the two sides. As the coefficients are the
Pontrjagin numbers of the manifolds involved, we get that the Pontrjagin
numbers of the two manifolds are all equal.

To finish the proof we have to repeat the whole argument using an analogous
 class instead of $\beta$, namely

\[ \beta'(\xi) = \prod_{i=1}^{\infty}(1+w_1(\xi)t_i^1+w_2(\xi)t_i^2+\dots)
\in H^*(B,\Z_2)[[t_1,t_2,\dots]]. \]

It is obvious that all the above hold for $\beta'$ as well. Thus not only the
Pontrjagin numbers, but also the Stiefel-Whitney numbers of the two manifolds
are equal. Since the oriented cobordism class is determined by these numbers,
the claim of the theorem follows.
\end{proof}

This result will no longer hold if we consider a general target space $N$.
However the Pontrjagin and Stiefel-Whitney numbers of the multiple-point
manifolds of $g_1 \times g_2$ are still expressible in terms of $g_1,g_2$ and
their multiple-point manifolds. This expression is particularly simple for the
double-point set.

First we need a simple result about the embedded manifold representing a vector
bundle's Euler class. Let $\xi \to B$ be a vector bundle over a manifold $B$.
Let $s: B \to \xi$ be a section transverse to the 0-section. Let us denote by
${\Delta_{\xi}}$ the submanifold in $B$ that is the inverse image of the
0-section by $s$, and let ${\delta_{\xi}} : {\Delta_{\xi}} \to B$ denote the
inclusion.

\begin{lem}
$\langle \beta({\Delta_{\xi}}),[{\Delta_{\xi}}]\rangle = \langle \beta(B) \cdot
\frac{e(\xi)}{\beta(\xi)}, [B]\rangle.$
\end{lem}

\begin{proof}
It suffices to show that  \[{\delta_{\xi}}_!(\beta({\Delta_{\xi}})) = \beta(B)
\cdot \frac{e(\xi)}{\beta(\xi)}.\]

By the construction of ${\Delta_{\xi}}$ we have the following pull-back
diagram:

\[
\begindc \obj(1,3){$\Delta_{\xi}$} \obj(3,3){$B$} \obj(1,1){$B$} \obj(3,1){$\xi$}
  \mor(1,3)(1,1){$\delta_{\xi}$} \mor(1,3)(3,3){$\delta_{\xi}$} \mor(1,1)(3,1){$s$}
  \mor(3,3)(3,1){$0$-section} \enddc
\]

Hence the normal bundle of ${\delta_{\xi}}$ is just the pull-back of the
normal-bundle of the 0-section. This latter is just $\xi$. Thus we have \[
T{\Delta_{\xi}} \oplus {\delta_{\xi}}^* \xi = {\delta_{\xi}}^* TB,\] which in
turn implies that \[ \beta({\Delta_{\xi}}) =
{\delta_{\xi}}^*\left(\frac{\beta(B)}{\beta(\xi)}\right).\] Applying the
push-forward to this equation gives the proof of the lemma, since $f_!(f^*x) =
f_!(1) \cdot x$ is well known and obviously ${\delta_{\xi}}_!(1) = e(\xi)$ .
\end{proof}

\begin{thm}~\label{2point}
Let $g_i : M_i^{n_i} \to N_i^{n_i+k_i}; (i=1,2)$ be generic immersions. Then
\[ \hu{M}{2}{g_1\times g_2} \sim \hu{M}{2}{g_1} \times
\hu{M}{2}{g_2} + \hu{M}{2}{g_1} \times \Delta_{\nu_{g_2}} + \Delta_{\nu_{g_1}}
\times \hu{M}{2}{g_2}\] where $\sim$ stands for ``unoriented-cobordant''.
(Recall that $\nu_{g_i}$ is the normal bundle of $g_i$ and $\Delta_{g_i}$ is
the zero set of a generic section of $\nu_{g_i}$.) If the $M_i$ are oriented
and the $k_i$ are even, then the same is true up to oriented cobordism.
\end{thm}

\begin{proof}
We proceed in a similar manner as in the previous theorem. Let us put $f =
g_1\times g_2$ and $M = M_1\times M_2$ again. Then using (\ref{mf}) we get

\begin{eqnarray*}
\beta(\nu_f)\cdot m_2(f) = && f^*f_!(\beta(M)) - e(\nu_f)\cdot \beta(M) = \\
=&& g_1^*{g_1}_!(\beta(M_1))\times g_2^*{g_2}_!(\beta(M_2))- e(\nu_f)\cdot
\beta(M) = \\ =&& \left(\beta(\nu_{g_1})m_2(g_1) + e(\nu_{g_1})\cdot
\beta(M_1)\right) \times \left(\beta(\nu_{g_2})m_2(g_2) + e(\nu_{g_2})\cdot
\beta(M_2)\right) -\\&&-e(\nu_f)\cdot \beta(M) =\\=&& \beta(\nu_f)\cdot
\left(m_2(g_1)\times m_2(g_2) + m_2(g_1) \times
\beta(M_2)\frac{e_{\nu_{g_2}}}{\beta(\nu_{g_2})} +
\beta(M_1)\frac{e_{\nu_{g_1}}}{\beta(\nu_{g_1})} \times m_2(g_2) \right)
\end{eqnarray*}

Now we can divide by $\beta(\nu_f)$ as it is an invertible element. We evaluate
both sides on $[M] = [M_1] \times [M_2]$. Finally we have to apply the previous
lemma to get that all the corresponding characteristic numbers are equal for
the two manifolds in question. As before, we can repeat the argument for
Stiefel-Whitney numbers in $\Z_2$ coefficients and Pontrjagin numbers in $\Q$
coefficients, so we get both parts of the theorem at the same time.
\end{proof}

\begin{rem}

\begin{enumerate}
\item It is possible to carry out similar calculations for triple points or
points of higher (say $r$) multiplicity. But the number of terms involved in
these formulas grow exponentially with $r$ and the authors did not manage to
find a nice way to write them down, not even recursively.

\item It would be possible to obtain similar formulas not only for the
cobordism classes of the underlying multiple-point manifolds, but for the
cobordism classes of the immersions $\phi_r$ themselves. To do this one would
need to consider the characteristic numbers of these immersions instead of the
characteristic numbers of the manifolds. These calculations are more or less
the same as the ones described here, but they are harder to keep track of.

\item It seems that the same results could be obtained using techniques of Eccles
and Grant from~\cite{EG}.

\item We would like to point out that Theorem~\ref{2point} is a non-trivial
generalisation of the oriented case of Theorem A in \cite{Byun}, which
considers the case of $n=k$.
\end{enumerate}
\end{rem}

\section{Ring structure of Morin maps}~\label{morinresz}

Given a smooth map $f : M \to N$, a point $x \in M$ is said to be a
$\Sigma^{i}$ point if the corank of $df_x : T_xM \to T_{f(x)}N$ is at least
$i$. The set of such points is denoted by $\Sigma^{i}(f)$. If $i_1 \geq i_2$
then we can define $\Sigma^{i_1,i_2}f = \Sigma^{i_2}f|_{\Sigma^{i_1}f}$. This
method can be continued recursively to give the definition of
$\Sigma^{(i_1,i_2,\dots,i_r)}$ points, where $i_1 \geq i_2 \geq \dots \geq
i_r$. This classification of singular points is called the Thom-Boardman type.
For details see e{.} g{.}~\cite{AGLV}.

A generic smooth map $f: M \to N$ is called a \textit{Morin} map if it has no
$\Sigma^2$ points. The singularities of such maps are classified by their
Thom-Boardman type, which can only be $\Sigma^{\overbrace{(1,1,\dots,1)}^r} =
\Sr{r}$ for some $r \geq 0$. (In the notation of~\cite{AGLV} this is $A_r$.)

Cobordism of Morin maps is defined in the usual way: two Morin maps $f : M_1^n
\to N^{n+k}$ and $g : M_2^n \to N^{n+k}$ are said to be cobordant if there is a
Morin map $H : W^{n+1} \to N^{n+k} \times [0,1]$ such that $\delta W = M_1 \cup
M_2$ and $H|_{M_1} = f, H|_{M_2} = g.$

Let us consider the set of cobordism classes of all Morin maps to Euclidean
spaces (for all nonnegative dimensions and all positive codimensions). This set
is a commutative group with addition induced by the disjoint union of maps. We
can take tensor product with $\Q$ to obtain the rational cobordism group whose
elements will be referred to as rational cobordism classes. In this section we
endow this rational cobordism group with a ring structure. Further we will show
that the singularities can be used to define ring homomorphisms into
$\Omega_*$, the oriented cobordism ring of manifolds.

The main tool in constructing the multiplication will be the so-called
``prim maps'', while the ring homomorphisms will be derived from the
results of the previous section.

\subsection{Prim maps}

\begin{dfn} A generic map $f: M \to N$ is called prim (\textit{pr}ojected
\textit{im}mersion) if it can be lifted to a generic immersion, $\tilde{f} : M
\to N \times \R$. (We will always denote the lifting by a tilde.)
\end{dfn}

Cobordism of prim maps can be defined in a natural way (the cobordism itself
should be a prim map into $N \times [0,1]$), and disjoint union induces a group
operation on the cobordism classes. The class of a prim map $f$ will be denoted
by $[f]$. (For details see e.g.~\cite{Sz2}.)

Clearly a prim map is neccessarily a Morin map.  Prim maps provide a good link
between immersions and Morin maps. We shall first define multiplication of prim
maps (using their liftings to immersions) and then show how to extend it to
multiplication of Morin maps (using results from~\cite{Sz3}). We will only work
with prim maps whose target space is Euclidean.

Let us denote $l_0 : pt \hookrightarrow \R$ the inclusion of a point into the
line.

\begin{lem}~\label{2vetulet}

a) Any two generic hyperplane projections of an immersion
  represent the same prim cobordism class.

b) Projections of cobordant immersions represent the same prim
cobordism class.
\end{lem}
\begin{proof}
a) Instead of taking two projections of the same immersion we can take
the same projection of two immersions which differ only by a
rotation. This rotation can be realized by a regular homotopy. We can
take a generic projection of this homotopy to a hyperplane that is
sufficiently close to the original one. This gives a prim cobordism
between slightly perturbed versions of the original prim maps, but
since generic projections form an open set this perturbation does not
effect the prim cobordism class (not even the prim homotopy class).
b) This can be proved in exactly the same way, by taking a generic
projection of the cobordism connecting the two immersions.
\end{proof}

\begin{dfn}~\label{primekszorzasa}
Given two prim maps $f_i : M_i \to \R^{n_i}\ (i=1,2)$ consider the product map
  \[g = f_1 \times f_2 \times l_0 : M_1 \times M_2 \to \R^{n_1+n_2}
  \times \R.\] The map $g$ might not yet be prim, but we can turn it
  into such by a small perturbation. Take liftings
  $\tilde{f_1}$ and $\tilde{f_2}$ that are sufficiently close to
  $f_1 \times l_0$ and $f_2 \times l_0$. Now $\tilde{f_1} \times
  \tilde{f_2} : M_1 \times M_2 \to \R^{n_1+n_2} \times \R^2$ is a
  non-generic immersion. Let us take a sufficiently small perturbation
  of this product so that it becomes a generic immersion. Finally take
  a generic projection this immersion to a hyperplane ``close'' to
  $\R^{n_1+n_2} \times \R$, where the last $\R$ factor is the diagonal
  in $\R^2$. We obviously get a prim map $g'$ that can be arbitrarily
  close to $g$. Let us denote $g' = f_1 * f_2$ and let us define the multiplication on prim cobordism classes as follows: $[f_1] * [f_2] =
  [f_1 * f_2]$.
\end{dfn}

\begin{thm} The above definition is correct, that is $[f_1 * f_2]$ is
  independent of the choice of $f_1$ and $f_2$ within their cobordism
  class and of any other choices made in the definition. The
  multiplication defined in this way gives rise to a ring structure
  with respect to the disjoint union as additon.
\end{thm}
\begin{proof} The liftings are uniqe up to regular homotopy. Also the
  perturbation of $\tilde{f_1} \times \tilde{f_2}$ is uniqe up to
  regular homotopy. Thus Lemma~\ref{2vetulet} implies that the
  resulting prim map is independent of these choices.

Now suppose $[f_1] = [g_1]$. Then there is a prim cobordism $H$
joining $f_1$ and $g_1$. We can take its lifting $\tilde{H}$ which is an
immersed cobordism between $\tilde{f_1}$ and $\tilde{g_1}$, and so
$\tilde{f_1} \times \tilde{f_2}$ and $\tilde{g_1} \times \tilde{f_2}$
are regularly homotopic via $\tilde{H} \times \tilde{f_2}$. So their
projections are prim cobordant, and this is what we wanted to
prove. (The definition is symmetric so the other factor can be handled
the same way.)

The last claim only requires the checking of distributivity, which is
obvious.
\end{proof}

\subsection{Morin maps}~\label{szorzas}

In this section we only consider maps between oriented manifolds. Let us denote
the group of cobordism classes of oriented Morin maps $f: M^n \to \R^{n+k}$ by
\Mor{n}{k} and the cobordism classes of prim maps $f: M^n \to \R^{n+k}$ by
\Prim{n}{k}. As a prim map is automatically Morin and prim cobordant maps are
Morin cobordant as well, we have a natural forgetting map $F : \Prim{n}{k} \to
\Mor{n}{k}$, that induces a map $F_{\Q} : \Prim{n}{k} \otimes \Q \to \Mor{n}{k}
\otimes \Q$. The following key result, which says that every Morin map has a
non-zero multiple that is Morin-cobordant to a prim map is proved
in~\cite{Sz3}:

\begin{lem}~\label{epi} The map $F_{\Q}$ is epimorphic.
\end{lem}

Using this result and the construction in the previous section we can
now define a multiplication on $\left(\bigoplus_{n,k} \Mor{n}{k}\right) \otimes \Q$.

\begin{dfn}~\label{morinekszorzasa} Let us take two Morin maps $g_i :
  M_i^{n_i} \to \R^{n_i+k_i}$. By Lemma~\ref{epi} we can find prim
  maps $f_1$ and $f_2$ that are rationally Morin cobordant to $g_1$
  and $g_2$. Let us define $[g_1] * [g_2] \defeq [F_{\Q}(f_1 * f_2)]$,
  where $[f]$ denotes the rational Morin cobordism class of the Morin
  map $f$.
\end{dfn}

\begin{thm} The above definition is correct, that is $[g_1] * [g_2]$
  is independent of the choices made. The multiplication defined this
  way gives rise to a ring structure on $\left(\bigoplus_{n,k} \Mor{n}{k}\right) \otimes \Q$.
\end{thm}

\begin{proof} There is only one thing left that needs to be checked:
  if $f_1$ and $f_1'$ are Morin cobordant prim-representatives of
  $g_1$, then $F(f_1 * f_2)$ is indeed Morin cobordant to
  $F(f_1'*f_2)$. Let us take the Morin cobordism $H$ connecting $f_1$ and
  $f_1'$. Then $H \times (f_2 \times l_0)$ is still a Morin cobordism
  after a sufficiently small perturbation, since the second factor can
  be perturbed to an immersion. This Morin cobordism connects exactly
  the two desired maps.
\end{proof}

\begin{dfn}
Let \MorQ denote the group $\bigoplus_{n,k} \Mor{n}{k} \otimes \Q$ with
this ring structure. \MorQ is a bigraded ring, the two grades being
$n$ and $k+1$.
\end{dfn}

\section{Computing \MorQ}~\label{morq}

\subsection{Ring homomorphisms}~\label{homom}

Let $k$ be odd, and let $f : M^n \to \R^{n+k}$ be a generic oriented
Morin map of odd codimension. To such a map we can associate the
subset of $M^n$ of those points where the Thom-Boardman singularity type
of $f$ is $\Sigma^{\overbrace{\mbox{\scriptsize 1,1,\dots,1}}^r} =
\Sr{r}$. This subset is actually a submanifold and will be denoted by
$\Sr{r}(f)$. The cobordism class of this submanifold is invariant
under a Morin cobordism of $f$, since the $\Sr{r}$ points of the
cobordism of $f$ give a cobordism between the $\Sr{r}$ points of
$f$. For even $r$ we actually get an oriented cobordism class. We can
tensor with $\Q$ and get a map \[\Sr{r} : \bigoplus_{k \mbox{
\scriptsize{odd}},n}\Mor{n}{k} \otimes \Q \to \Omega_* \otimes \Q\] to
the rational oriented cobordism ring.

\begin{thm}
If $r$ is even then the map $\Sr{r}$ is a ring homomorphism or in other words
  for Morin maps $f,g$ to Euclidean spaces we have \[ \Sr{r}(f * g)
  \sim \Sr{r}(f) \times \Sr{r}(g) \] where $\sim$ now stands for
  rationally cobordant (in the oriented sense).
\end{thm}

\begin{proof}
We will proceed along the lines explained earlier, that is we will use
prim maps as a link between Morin maps and immersions. Then the
multiplicative properties of multiple points of immersions will
provide the result.

Let us first consider prim maps. The same argument as above gives a
map \[\Sr{r}_{Pr} : \left(\bigoplus_{k \mbox{ \scriptsize{odd}},n}\Prim{n}{k}\right)\otimes \Q \to
\Omega_* \otimes \Q. \] It is obvious that $\Sr{r}_{Pr} = \Sr{r}
\circ F_{\Q}$.

 Let us denote the oriented cobordism groups of $k+1$ codimensional
 immersions from $n$-dimensional manifolds to Euclidean spaces by
 $\mbox{Imm}^{SO}(n,k+1)$. Given an immersion $f : M^n \to
 \R^{n+k+1}$, let us denote by $\pi(f)$ its generic projection to a
 hyperplane. This map is a prim map whose prim cobordism class is well
 defined according to Lemma~\ref{2vetulet}. The direct sum $\bigoplus_{k \mbox{
 \scriptsize{odd}},n} \mbox{Imm}^{SO}(n,k+1)$ has a natural ring structure
 with multiplication being the direct product. It is clear from the
 definitions that
\[ \pi : \bigoplus_{k \mbox{ \scriptsize{odd}},n}\mbox{Imm}^{SO}(n,k+1) \to \bigoplus_{k \mbox{ \scriptsize{odd}},n}\Prim{n}{k}\]
is a ring homomorphism with respect to the direct product on the left,
and $*$-product on the right. The same remains true after forming the
tensor product with $\Q$.

In Theorem~\ref{immszorz} we have shown that
\[ \tilde{M}_{r+1} : \bigoplus_{k \mbox{ \scriptsize{odd}},n}\mbox{Imm}^{SO}(n,k+1) \to \Omega_*\]
  is a ring homomorphism, and obviously the same is true after forming
  the tensor product with $\Q$.

To finish the proof we have to recall a result from~\cite{Sz1} which
in our notations reads as:

\begin{thm}[\cite{Sz1}]~\label{szigmadelta} $\tilde{M}_{r+1}
  \otimes \id_{\Q}
  = (\pi \otimes \id_{\Q}) \circ \Sr{r}_{Pr}$ \ie the rational
  cobordism class of the manifold of $r+1$-tuple points of an
  immersion $f: M^n \to R^{n+k+1}$ coincides with that of the manifold
  of $\Sr{r}$ (or $A_r$) points of its hyperplane projection.
\end{thm}

All of the above proves that the following diagram is commutative.

\[
\begindc[3]
\obj(1,50){$\left(\bigoplus_{k \mbox{ \scriptsize{odd}},n}\mbox{Imm}^{SO}(n,k+1)\right)\otimes \Q $}
\obj(1,25){$\left(\bigoplus_{k \mbox{ \scriptsize{odd}},n}\Prim{n}{k}\right)\otimes \Q $}
\obj(45,25){$\Omega_* \otimes \Q$} \obj(1,1){$\left(\bigoplus_{k
\mbox{ \scriptsize{odd}},n}\Mor{n}{k}\right)\otimes \Q $}
\mor(9,47)(40,27){\scriptsize$\tilde{M}_{r+1} \otimes \id_{\Q}$}
\mor(1,50)(1,27){\scriptsize$\pi \otimes \id_{\Q}$}
\mor(1,25)(1,3){\scriptsize$F_{\Q}$}
\mor(20,25)(40,25){\scriptsize$\Sr{r}_{Pr}$}
\mor(9,5)(40,23){\scriptsize$\Sr{r}$}
\enddc
\]

The vertical maps are ring epimorphisms and $\tilde{M}_{r+1}$ is a
ring homomorphism. This implies that $\Sr{r}_{Pr}$ and $\Sr{r}$ are
ring homomorphisms too.
\end{proof}

\subsection{The structure of $\Mor{n}{k}$}~\label{direktosszeg}

In~\cite{Sz3} it is shown that the rational cobordism class of an
oriented Morin map is actually determined by those of its singular
strata. As we have seen the singular strata are ring homomorphisms
from \MorQ. This provides a complete computation of the ring \MorQ.


For any stable singularity type $\eta$ there is a bundle $\xih_\eta$
that plays the role of the universal normal bundle for this
singularity type. This means the following: Whenever for a map $f:M\to
N$ one of its most complicated singularities is $\eta$ then the
$\eta$-points of $f$ form a submanifold of $M$. The restriction of $f$
to this submanifold is an immersion to $N$. The normal bundle of this
immersion is induced from $\xih_\eta$. (See~\cite{RSz} for details.)

Let us write $\xih_r = \xih_{\Sigma^{1_r}}$ for short. Let
$\mbox{Imm}^{\xih_r}(n,k)$ denote the cobordism group of oriented
immersions $f: M^n \to \R^{n+k}$ whose normal bundles are induced from
$\xih_r$.

We need two results from~\cite{Sz3} which we state here in a lemma.

\begin{lem}~\label{szucs}
\begin{enumerate}
\item For odd $k$ we have
  \begin{equation} \Mor{n}{k}\otimes \Q = \bigoplus_{i = 0}^{\infty}
    \mbox{Imm}^{\xih_{2i}}(n-2i(k+1),2i(k+1)+k)\otimes
    \Q.\label{split}
  \end{equation}
while for even $k$ we have $\Mor{n}{k} \otimes \Q = \mbox{Imm}^{SO}(n,k)$.
\item For even $r$ we have $H_{n+k}(T\xih_r;\Q) =
H_{n-r(k+1)}(BSO(k);\Q)$.
\end{enumerate}
\end{lem}

\begin{proof}
Part (i) is stated explicitly in~\cite{Sz3} as Example 119.

For part (ii) we have to recall that the bundle $\xih_{\eta}$ has a
counterpart denoted by $\xi_{\eta}$ which is the universal normal
bundle of the $\eta$-points of a map in the source manifold. The two
bundles $\xi_{\eta}$ and $\xih_{\eta}$ have the same base space
$BG_{\eta}$ where $G_{\eta}$ is the maximal compact subgroup of the
symmetry group of the singularity ${\eta}$. This implies that the
homologies of $T\xih_{\eta}$ and $T\xi_{\eta}$ are the same up to a
dimension shift equal to $\rank \xih_{\eta} - \rank \xi_{\eta} = k$,
\ie $H_{n+k}(T\xih_r;\Q) = H_n(T\xi_r;\Q)$.

Lemma 103/b in~\cite{Sz3} implies that for even $r$ we have
$H_{n}(T\xi_r;\Q) = H_{n-r(k+1)}(BSO(k);\Q)$. The statement follows.
\end{proof}

It is well known that \[\mbox{Imm}^{\xih_r}(n,k) \otimes \Q \cong
\pi_{n+k}^S(T\xih_r)\otimes \Q \cong H_{n+k}(T\xih_r;\Q) =
H_{n-r(k+1)}(BSO(k);\Q).\] There is the natural forgetting map that
assigns to an immersion the cobordism class of its underlying source
manifold. This forgetting map on the level of classifying spaces is
just the inclusion of the classifying spaces $BSO(k) \hookrightarrow
BSO$. The rational cohomology ring of the classifying space for
$\Omega_*$ is $\Q[p_1,p_2,\dots]$. Since $k$ is odd $H^*(BSO(k);\Q) =
\Q[p_1,p_2,\dots,p_{\frac{k-1}{2}}]$. Thus the inclusion map induces a
surjective homomorphism between the rings and this means that the
forgetting map is actually injective.

Thus for every even $r$ we have a map $\Mor{n}{k} \otimes \Q \to
\mbox{Imm}^{\xih_r}(n-r(k+1),r(k+1)+k)\otimes \Q \to \Omega_{n-r(k+1)}
\otimes \Q$. The first arrow is just the projection in the
splitting~{(\ref{split})} while the second arrow is the forgetting
map. The composition of the two is obviously the previously defined
$\Sigma^{1_r}$.

This proves that for odd $k$ an element $[f] \in \Mor{n}{k}\otimes \Q$
is indeed determined by the collection of rational cobordism classes
of the $\Sigma^{1_r}f$ manifolds. It also follows from the previous
argument that exactly those cobordism classes are in the image
$\Sigma^{1_r}(\Mor{n}{k} \otimes \Q)$ which do not have non-zero
Pontrjagin numbers involving Pontrjagin classes higher than
$p_{\frac{k-1}{2}}$.

For even $k$ the situation is simpler. It follows from
Lemma~\ref{szucs} that for an element $[f] \in \Mor{n}{k} \otimes \Q$ we
have $\Sigma^{1_r}(f) = 0$ for every $r \geq 1$ and thus the class of
$f$ is completely determined by the cobordism class of its underlying
manifold. In other words any even codimensional Morin map is
Morin-cobordant to an immersion. It is then clear from the
definitions~\ref{primekszorzasa} and~\ref{morinekszorzasa} that
multiplying by an even codimensional map annihilates any singularities.

\section{Singular strata of direct products}~\label{szigma1}

Our goal in this final section is to show that the cohomology class represented
by the submanifold formed by the closure of the set of certain singular points
of a direct product $f \times g$
depends only on those $f$ and $g$ and some maps
closely related to them.

The arguments are based on the well known fact, that the
Thom polynomials of the singularity types in question are simple.
Before we formulate the theorems, we have to introduce some notation.

\begin{dfn}
For $j \geq 0$ let $q_j : * \to S^j$ denote the inclusion of a point into $S^j$
and for $j < 0$ let $q_j : S^{|j|} \to *$ be the map that takes the sphere to a
point. Now for any integer $j$ we define $f_j' = f \times q_j$ and take
$f_j$ to be a generic perturbation of $f_j'$.

Finally let $\id_j = \id_M \times q_j$.
\end{dfn}

\subsection{The $\Sigma^1$ stratum}

Let $\Sigma^1 f$ denote the closure of the set of all singular points in
the source manifold of $f$. The Thom polynomial of this singularity
type is $w_{k+1}$. That is, given a map $f: M^n \to N^{n+k}$ , the
cohomology class Poincar\'e dual to the homology class represented by
$\Sigma^1 f$ is equal to $w_{k+1}(\nu_f)$ where $\nu_f$ stands for the
virtual normal bundle of $f$.  This dual cohomology class will be
denoted by $[\Sigma^1 f]$ for simplicity.

\begin{thm}
Let $f : M_1^{n_1} \to N^{n_1+k_1}, g : M_2^{n_2} \to N_2^{n_2+k_2}$ be two
generic maps. Then for a generic perturbation of their product we have
\[ [\Sigma^1 {f\times g}] = \sum_{j\geq 1} \left(\vphantom{\sum}[\Sigma^1 {f_{j-1}}] \times
\id_j^* [\Sigma^1 {g_{(-j)}}] + \id_j^* [\Sigma^1 {f_{(-j)}}] \times [\Sigma^1
{g_{j-1}}]\right)\]
\end{thm}

\begin{proof}
As a first step let us notice that since $\nu_{f\times g} = \nu_f \times \nu_g$
we can write
\begin{multline*} w_{k_1+k_2+1}(\nu_{f\times g}) = \sum_{r = 0}^{k_1+k_2+1}w_r(\nu_f) \times
w_{k_1+k_2+1-r}(\nu_g) =\\ =  \sum_{j \geq 1} \left( \vphantom{\sum}
w_{k_1+j}(\nu_f)\times w_{k_2-j+1}(\nu_g) + w_{k_1-j+1}(\nu_f)\times
w_{k_2+j}(\nu_g)\right)
\end{multline*}

Now we have to take a closer look at $w_{k_1+j}(\nu_f)$. If $k_1+j-1$ would be
equal to the codimension of $f$ then this characteristic class would just
represent the singular locus of $f$. When this is not the case, we have to find
an appropriate replacement of $f$ that has the right codimension, whose normal
bundle however is stably equivalent to that of $f$. This replacement map is
exactly $f_{j-1}$. Indeed, $\nu_{f_{j-1}} = \nu_f \oplus \varepsilon^{j-1}$ so
$w_{k_1+j}(\nu_f) = w_{k_1+j}(\nu_{f_{j-1}})$ which in turn is equal to
$[\Sigma^1 f_{j-1}]$ since this map has the right codimension.

The argument is just slightly more complicated in the case of $w_{k_2-j+1}$.
Here first we take the map $g_{(-j)} : M_2^{n_2} \times S^j \to N_2^{n_2+k_2}$.
This has codimension $k_2-j$ so $[\Sigma^1 g_{(-j)}] =
w_{k_2-j+1}(\nu_{g_{(-j)}})$. The only problem is that this class lives in the
cohomology of $M_2 \times S^j$. This is why we have to pull it back to $M_2$ by
$\id_{(-j)}$. Since the composition of $id_j$ and $g_{(-j)}$ is just a
perturbation of $g$ and $w(\nu_{q_j}) = 1$ it follows that $id_j^*
w_{k_2-j+1}(\nu_{g_{(-j)}}) = w_{k_2-j+1}(\nu_g)$.

Putting all these together gives the result of the theorem.
\end{proof}

\subsection{The $\Sigma^2$ stratum}

A very similar result can be proved about the $\Sigma^2$ stratum
of oriented maps. First we need to compute the Thom polynomial of
the $\Sigma^2$ stratum in the oriented case. We will work with
rational coefficients.

\begin{thm}Let $f: M^n \to N^{n+k}$ be a generic map where
$(k=2t-2)$. Then the rational cohomology class dual to the closure of the set
of $\Sigma^2$-points of $f$ (for short $[\Sigma^2 f]$) equals $p_t(\nu_f)$,
where $p_t \in H^{4t}(M;\Q)$ is the $t^{\mbox{th}}$ Pontrjagin class.
\end{thm}

\begin{proof}
By definition the Thom polynomial $\tps$ of the $\Sigma^2$-stratum is
a cohomology class in $H^{4t}(BSO;\Q) = \Q[p_1,p_2,p_3,\dots]$.  We
want to show that $\tps = p_t$. It is enough to show that these two
cohomology classes evaluated on any homology class
in $H_{4t}(BSO;\Q)$ are equal.

\begin{lem} All homology classes in $H_{4t}(BSO;\Q)$ can be
represented by a normal map, \ie by a map $h: L^{4t} \to BSO$ of an
oriented $4t$-manifold $L^{4t}$ corresponding to the stable normal bundle of
$L^{4t}$.
\end{lem}

\begin{proof}
 It is enough to consider a sufficiently large finite
dimensional approximation $BSO(N),\, (N \gg 1)$. By the Pontrjagin-Thom
construction an embedding $L^{4t} \hookrightarrow S^K$ gives a map
$h': S^K \to MSO(K-4t)$ that maps $L^{4t}$ into $BSO(K-4t)$ and the
restriction $h'|_{L^4t}$ corresponds to the normal bundle of $L^{4t}$.
The homotopy class $[h'] \in \pi_K(MSO(K-4t))$ is mapped by the
composition of the Hurewicz homomorphism and the Thom isomorphism into
a homology class $x = h'_*([L^{4t}]) \in H_{4t}(BSO(K-4t))$. Hence this
class $x$ is represented by a normal map. Since the Hurewicz
homomorphism in stable dimensions $(K \geq 8t+2)$ is a rational
isomorphism, we obtain the statement of the lemma.
\end{proof}

To evaluate a $4t$ dimensional cohomology class on a $4t$ dimensional homology
class represented by a manifold, one just pulls back the cohomology class to
the manifold and evaluates it on the fundamental class.

Now it is enough to prove, that for every oriented $M^{4t}$ the map $\nu^* :
H^{4t}(BSO;\Q) \to H^{4t}(M;\Q)$ induced by the normal mapping $\nu : M^{4t}
\to BSO$ takes $p_t$ and $\tps$ to the same cohomology class in $H^{4t}(M;\Q)$.
As $\nu^*(p_t) = p_t(\nu_M)$ and $\nu^*(\tps)$ is the dual of the $\Sigma^2$
stratum of a generic map $M^{4t} \to \R^{6t-2}$ we reduced the problem of
finding the Thom polynomial to the special case of $M^{4t} \to \R^{6t-2}$ maps.

If we take an immersion $f: M^{4t} \to \R^{6t}$, and project it to two
non-parallel hyperplanes, then we get a map $f' : M^{4t} \to
\R^{6t-2}$. Let us denote the two hyperplanes $H_1,H_2$. The
projection of $f$ to $H_i$ shall be called $f_i$. It is obvious that
those and only those points belong to $\Sigma^2 f'$ which belong to
$\Sigma^1 f_1$ and $\Sigma^1 f_2$ at the same time. This means that
for this $f'$ we have $[\Sigma^2 f'] = [\Sigma^1 f_1] \cup [\Sigma^1
f_2]$. The two cohomology classes on the right are both equal to the
Thom polynomial of the $\Sigma^1$ singularity, which is the Euler
class of the normal bundle of $f$. As this normal bundle has rank
$2t$, the square of its Euler class is equal to $p_t(\nu_f)$, which is
the same as $p_t(\nu_M)$. So far we have proved our claim for those
maps $M^{4t} \to \R^{6t-2}$ where the source manifold can be immersed
into $\R^{6t}$.

Let us recall that by $\mbox{Imm}^{SO}(4t,2t)$ we denoted the
cobordism group of oriented immersions from $4t$ dimensional manifolds
to $\R^{6t}$. There is the natural forgetting map $\psi:
\mbox{Imm}^{SO}(4t,2t) \to \Omega_{4t}$ taking an immersion to its
underlying manifold. To finish the proof of the theorem it is
sufficient to show, that this map is a rational epimorphism. According
to the Pontrjagin-Thom construction and the stable Hurewicz
homomorphism
\[\mbox{Imm}^{SO}(4t,2t) \cong \pi^S_{6t}MSO(2t)
\stackrel{\mbox{\scriptsize{$\Q$}}}{\cong} H_{6t}(MSO(2t);\Q))\] and
\[\Omega_{4t} \cong \pi^S_{4t}(MSO) \stackrel{\mbox{\scriptsize{$\Q$}}}{\cong}
H_{4t}(MSO;\Q),\] where $\stackrel{\mbox{\scriptsize{$\Q$}}}{\cong}$
means isomorphic if tensored with $\Q$. Thus $\psi$ being epimorphic
is equivalent to
\[\psi_H :H_{6t}(MSO(2t);\Q)) \to H_{4t}(MSO;\Q)\] being epimorphic,
which is further equivalent to (by taking the dual morphism in
cohomology) \[\psi^* : H^{4t}(MSO;\Q) \to H^{6t}(MSO(2t);\Q))\] being
monomorphic. We can apply the Thom-isomorphism to further reduce the
problem to showing that
\[ \psi^*_B : H^{4t}(BSO;\Q) \to H^{4t}(BSO(2t);\Q)\] is monomorphic. It is
easy to see that $\psi^*_B$ is induced by the natural inclusion map
$BSO(2t) \hookrightarrow BSO$. The cohomology ring of $BSO(2t)$ is the
polynomial ring $\Q[p_1,p_2,\dots,p_{t-1},\chi_{2t}]$ generated by the
Pontrjagin classes and the Euler class, whose square is $p_t$. On the
other hand $H^*(BSO;\Q) \cong \Q[p_1,p_2,\dots]$. As $\psi^*_B$ takes
each Pontrjagin class to the same Pontrjagin class, we get that
$\psi^*_B$ is indeed injective in dimension $4t$. This completes the
proof of $\tps = p_t$.
\end{proof}

When we want to consider direct products of maps, we will need the
Cartan formula. For Pontrjagin classes the Cartan formula only holds
mod 2, so we will need to consider everything in $H^*(M;\Q)$ to get
rid of the 2-torsion.

The proof of the next theorem copies the proof of the previous section.

\begin{thm}
Let $f : M_1^{n_1} \to N^{n_1+k_1}, g : M_2^{n_2} \to N_2^{n_2+k_2}$
be two generic maps of even codimension. Then for a generic
perturbation of their product we have
\[ [\Sigma^2 {f\times g}] = \sum_{j\geq 1} \left(\vphantom{\sum}[\Sigma^2 {f_{2j-2}}] \times
\id_{2j}^* [\Sigma^2 {g_{(-2j)}}] + \id_{2j}^* [\Sigma^2 {f_{(-2j)}}]
\times [\Sigma^2 {g_{2j-2}}]\right)\]
\end{thm}


\begin{thebibliography}{5}

\bibitem{AGLV}
{\bibname V. I. Arnol'd, V. V. Goryunov, O. V. Lyashko \and  V. A. Vassiliev},
'Singularities I. Local and global theory', \emph{Encyclopaedia of
  Mathematical Sciences} vol.6. Dynamical Systems VI. (Springer Verlag,
Berlin, 1993.)

\bibitem{BL}
{\bibname G. Braun \and G. Lippner}, 'Characteristic numbers of multiple-point
manifolds',
\emph{Bull.\ London Math.\ Soc.\ } 38,
  No. 4 (2006) 667-678.

\bibitem{Byun}
{\bibname Y. Byun \and S. Yi}, 'Product formula for self-intersection
  numbers',
\emph{Pac.\ J.\ Math.\ } 200, No. 2 (2001) 313-330.

\bibitem{EG}
{\bibname P. Eccles \and M. Grant}, 'Bordism groups of immersions represented
by self-intersections',
\emph{Alg.\ \& Geom.\ Top.\ } 7 (2007) 1081-1097.

\bibitem{RSz}
{\bibname R. Rim\'anyi \and A. Sz\H ucs},
'Pontrjagin - Thom type
  construction for maps with singularities',
\emph{Topology}
  37 (1998) 1177-1191.

\bibitem{Sz2}
{\bibname A. Sz\H ucs},
'On the cobordism group of immersions and
  embeddings',
\emph{Math Proc.\ Camb.\ Phil.\ Soc.\ } 109 (1981) 343-349.

\bibitem{Sz1}
{\bibname A. Sz\H ucs}, 'On the singularities of hyperplane projections of
immersions', \emph{Bull.\ London Math.\ Soc.\ } 32 (2000) 364-374.

\bibitem{Sz3}
{\bibname A. Sz\H ucs}, 'Cobordism of singular maps', preprint, arXiv:math/0612152.

\end{thebibliography}
\end{document}